\documentclass[12pt,reqno]{amsart}
\usepackage{fullpage}
\usepackage{times}
\usepackage{colonequals}
\usepackage{amsmath,amssymb,amsthm,url}
\usepackage{xcolor}
\usepackage[utf8]{inputenc}
\usepackage[english]{babel}
\usepackage{comment}
\usepackage{bbm}
\usepackage{enumerate}
\usepackage{bm}
\usepackage{graphicx}
\usepackage{mathrsfs}
\usepackage[colorlinks=true, pdfstartview=FitH, linkcolor=blue, citecolor=blue, urlcolor=blue]{hyperref}

\newtheorem{thm}{Theorem}[section]
\newtheorem{lem}[thm]{Lemma}

\newtheorem{claim}[thm]{Claim}

\theoremstyle{definition}

\newcommand{\N}{\mathbb{N}}
\newcommand{\Z}{\mathbb{Z}}

\newcommand{\bx}{\mathbf{x}}

\newenvironment{poc}{\begin{proof}[Proof of claim]}{\end{proof}}

\title{An Erd\H{o}s--Ko--Rado theorem for binary codes}
\author{Shamil Asgarli}
\address{Department of Mathematics \& Computer Science \\ Santa Clara University \\ Santa Clara, CA 95053 \\ United States}
\email{sasgarli@scu.edu}

\author{Chi Hoi Yip}
\address{School of Mathematics\\ Georgia Institute of Technology\\ Atlanta, GA 30332\\ United States}
\email{cyip30@gatech.edu}

\begin{document}

\begin{abstract}
We study intersecting families of words from the Erd\H{o}s--Ko--Rado perspective. When the alphabet size is $2$, a maximum intersecting family is not necessarily a star. However, we prove that every maximum $3$-wise intersecting family is a star. We also present a new proof of the known result for alphabets of size at least $3$: maximum intersecting families of words are exactly the stars.
\end{abstract}

\maketitle

\section{Introduction}

The Erd\H{o}s--Ko--Rado (EKR) theorem \cite{EKR} is a cornerstone result in extremal combinatorics. The theorem states that if $n$ and $k$ are positive integers with $n \geq 2k$, then the maximum size of an intersecting family of $k$-element subsets of an $n$-element set is $\binom{n-1}{k-1}$. Furthermore, when $n > 2k$, the extremal families are precisely the ``stars'', consisting of all $k$-element subsets that share a fixed common element.

Since its inception, the philosophy of the EKR theorem has blossomed into an active research field. The main goal is to characterize the size and structure of maximum intersecting families. This line of inquiry has been successfully ported to structures beyond set families. Analogous EKR-type theorems exist for vector spaces, permutations, orthogonal arrays, and other algebraic structures. For a comprehensive overview of these generalizations, especially through algebraic graph theory and association schemes, we refer the reader to the excellent book by Godsil and Meagher \cite{GM16}. There are also well-studied variations. These include $r$-wise intersecting families (any $r$ members intersect at the same point), $t$-intersecting families (any two intersect in at least $t$ points), or a combination of both.

In this paper, we focus on EKR-type questions for words. Let $X$ be a finite alphabet of size $q \geq 2$, and let $m \in \N$. We consider the set $X^m$ of all words of length $m$. Two words $\bx = (x_1, \dots, x_m)$ and $\mathbf{y} = (y_1, \dots, y_m)$ in $X^m$ are said to \emph{intersect} if they agree in at least one coordinate, that is, if $x_j = y_j$ for some $1 \leq j \leq m$. A collection $T \subseteq X^m$ is an \emph{intersecting family} if every pair of words in $T$ intersects. The prototypical example of a large intersecting family is a \emph{star} (or \emph{canonical intersecting family}), which is defined by fixing a specific letter at a specific coordinate. All stars in $X^m$ have cardinality $q^{m-1}$, and it is well-known that they have maximum size; for the sake of completeness, we provide a short proof in the following lemma.

\begin{lem}\label{lem:weak}
If $T \subseteq X^m$ is an intersecting family of words of length $m$ on an alphabet $X$ of size $q\geq 2$, then $|T|\leq q^{m-1}$.
\end{lem}
\begin{proof} Without loss of generality, $X$ is given by $\Z_q \colonequals \mathbb{Z}/q\mathbb{Z}$. For each $\boldsymbol{\delta}\in \{0\} \times \Z_q^{m-1}$, consider the set $S(\boldsymbol{\delta})=\{\boldsymbol{\delta}, \boldsymbol{\delta}+\mathbf{1}, \dots, \boldsymbol{\delta}+(\mathbf{q-1})\}$ consisting of $q$ words, where $\mathbf{c}=(c, \dots, c)$ represents a constant vector in $\Z_q^{m}$. As we vary $\boldsymbol{\delta}$, these $q^{m-1}$ resulting sets $S(\boldsymbol{\delta})$ partition $\Z_q^m$. Observe that for each $\boldsymbol{\delta}\in \{0\} \times \Z_q^{m-1}$, we have $|T \cap S(\boldsymbol{\delta})|\leq 1$ since $T$ is intersecting. Thus, $|T|\leq q^{m-1}$.
\end{proof}

The Erd\H{o}s--Ko--Rado phenomenon for words (equivalently, integer sequences or the Hamming scheme setting) has a rich history. Early extremal results on binary and integer sequences in Hamming-type spaces go back to Kleitman \cite{Kle66}, and explicit EKR-type results for words were established in the 1980s by Livingston \cite{L79}, Frankl--F\"uredi \cite{FF80}, Moon \cite{Moo82}, Gronau \cite{Gro83}, and Engel--Frankl \cite{EF86}. Later, Frankl--Tokushige \cite{FT99} resolved a conjecture from \cite{FF80}. For a modern account of EKR-type results on words, see \cite[Chapter 10]{GM16}.

Previous works on the subject have focused mostly on the case where $q\geq 3$. In this range, the stars are the \emph{unique} maximum intersecting families, a result first proved by Livingston \cite{L79}.

\begin{thm}\label{thm:alphabet-intersecting}
Let $X$ be a finite alphabet with $|X|=q\geq 3$ and let $m\in\mathbb{N}$. Every maximum intersecting family $T$ of words of length $m$ is a star.
\end{thm}

Note that Theorem~\ref{thm:alphabet-intersecting} fails when $q=2$ (that is, for binary codes). In fact, there are precisely $2^{2^{m-1}}$ maximum intersecting families. To see this, partition $\{0,1\}^m$ into $2^{m-1}$ complementary pairs, where each pair sums to the constant vector $\mathbf{1}$. Choosing one word from each pair yields an intersecting family of size $2^{m-1}$, and therefore a maximum intersecting family. Conversely, every intersecting family of size $2^{m-1}$ arises in this way, since the two words in each complementary pair do not intersect. By contrast, there are only $2m$ stars in $\{0,1\}^m$, so the vast majority of maximum intersecting families are not stars.

To the best of our knowledge, no previous work has further examined the case $q=2$. One way to remedy the conclusion of Theorem~\ref{thm:alphabet-intersecting} for binary words is to assume a stronger hypothesis on the intersecting family. A collection $T$ of binary words is \emph{3-wise intersecting} if every three elements of $T$ share at least one common bit at the same index. Our main result shows that this stronger condition is indeed sufficient to guarantee that the family is a star.

\begin{thm}\label{thm:alphabet-3-wise-intersecting}
Every maximum $3$-wise intersecting family of binary words in $\{0, 1\}^m$ is a star.
\end{thm}

We prove Theorem~\ref{thm:alphabet-3-wise-intersecting} in Section~\ref{sec:3-wise}. Inspired by the proof of Theorem~\ref{thm:alphabet-3-wise-intersecting}, we also give a new proof of Theorem~\ref{thm:alphabet-intersecting} in Section~\ref{sec:q>2}. 

\section{Proof of Theorem~\ref{thm:alphabet-3-wise-intersecting}}\label{sec:3-wise}

We outline the strategy before giving the details. We argue by contradiction, assuming $T$ is not a star. The key step is Claim~\ref{claim}, which shows that the $3$-wise intersecting hypothesis forces each binary prefix of length $k$ to appear exactly $2^{m-k-1}$ times in $T$. Applying this with $k=m-1$ yields two words $\mathbf{x}$ and $\mathbf{y}$ in $T$ sharing the same last bit $u$, and the $3$-wise intersecting condition then forces every word in $T$ to have last coordinate $u$; this makes $T$ a star, a contradiction.

\begin{proof}[Proof of Theorem~\ref{thm:alphabet-3-wise-intersecting}]
Let $T$ be a maximum $3$-wise intersecting family. By Lemma~\ref{lem:weak}, $|T|\leq 2^{m-1}$. On the other hand, a star has size $2^{m-1}$, and hence $|T|=2^{m-1}$. Assume for contradiction that $T$ is not a star. 

For each $0\leq k\leq m$, we can partition $T$ according to the first $k$ bits. More precisely, given $\boldsymbol{\delta}=(\delta_1, \dots, \delta_k) \in\{0,1\}^k$, we define 
\[
T_{\boldsymbol{\delta}} = \{\mathbf{x}\in T \ | \ x_i = \delta_i \text{ for each } 1\leq i\leq k \}.
\]
When $k=0$, we define $T_{\boldsymbol{\delta}}=T$. Then $T$ is partitioned as
\[
T = \bigcup_{\boldsymbol{\delta}\in\{0,1\}^{k}} T_{\boldsymbol{\delta}}.
\]

Our key observation is the following claim.
\begin{claim}\label{claim} 
Assume that $T$ is $3$-wise intersecting and not a star. For each $0\leq k\leq m-1$ and each $\boldsymbol{\delta} \in \{0,1\}^k$, we have
\begin{equation}\label{eq:binary-vector-induction}
|T_{\boldsymbol{\delta}}| = 2^{m-k-1}.
\end{equation}
Equivalently, among the words of $T$, each prefix of length $k$ appears exactly $2^{m-k-1}$ times.
\end{claim}
\begin{poc}
We prove the claim by induction on $k$. When $k=0$, the assertion is trivial. 

Fix $k\geq 1$. First, we establish that $T_{\boldsymbol{\delta}}\neq \emptyset$ for each $\boldsymbol{\delta} \in \{0,1\}^k$. Suppose, to the contrary, that $T_{\boldsymbol{\alpha}}=\emptyset$ for some $\boldsymbol{\alpha}=(\alpha_1, \dots, \alpha_k)$. By induction, 
$$
|T_{(\alpha_1, \dots, \alpha_{k-1})}|= 2^{m-(k-1)-1} = 2^{m-k}.
$$ 
After decomposing the set above into a disjoint union
\[
T_{(\alpha_1, \dots, \alpha_{k-1})} = T_{(\alpha_1, \dots, \alpha_{k-1}, \alpha_k)} \cup T_{(\alpha_1, \dots, \alpha_{k-1}, 1-\alpha_k)}
\]
and using $T_{\boldsymbol{\alpha}}=\emptyset$, we see that $|T_{(\alpha_1, \dots, \alpha_{k-1}, 1-\alpha_k)}| = 2^{m-k}$ and thus
\[
T_{(\alpha_1, \dots, \alpha_{k-1}, 1-\alpha_k)}=\{(\alpha_1, \dots, \alpha_{k-1}, 1-\alpha_k, x_{k+1}, \ldots, x_m): x_j \in \{0,1\} \text{ for } k+1\leq j \leq m\}.
\]
Suppose 
\[
T_{(1-\alpha_1, 1-\alpha_2, \dots, 1-\alpha_{k-1}, \alpha_k)}\neq \emptyset.
\]
Pick $(1-\alpha_1, 1-\alpha_2, \dots, 1-\alpha_{k-1}, \alpha_k, y_{k+1}, \ldots, y_m) \in T_{(1-\alpha_1, 1-\alpha_2, \dots, 1-\alpha_{k-1}, \alpha_k)}$. Then \[(\alpha_1, \alpha_2, \dots, \alpha_{k-1}, 1-\alpha_k, 1-y_{k+1}, \ldots, 1-y_m) \in T_{(\alpha_1, \dots, \alpha_{k-1}, 1-\alpha_k)}\]
but 
$$
(1-\alpha_1, 1-\alpha_2, \dots, 1-\alpha_{k-1}, \alpha_k, y_{k+1}, \ldots, y_m) \text{ and }(\alpha_1, \alpha_2, \dots, \alpha_{k-1}, 1-\alpha_k, 1-y_{k+1}, \ldots, 1-y_m)
$$
do not intersect, a contradiction. Thus,
\[
T_{(1-\alpha_1, 1-\alpha_2, \dots, 1-\alpha_{k-1}, \alpha_k)}= \emptyset.
\]
By a similar argument as above, we have
\[
T_{(1-\alpha_1, 1-\alpha_2, \dots, 1-\alpha_k)}=\{(1-\alpha_1, \dots, 1-\alpha_k, x_{k+1}, \ldots, x_m): x_j \in \{0,1\} \text{ for } k+1\leq j \leq m\}.
\]
Since $T$ is not a star, it is \emph{not} the family of all binary vectors whose $k$-th bit is $1-\alpha_k$. Therefore, we can choose $\boldsymbol{\rho} = (\rho_1, \dots, \rho_k, \dots, \rho_m)\in T$ where $\rho_k=\alpha_k$. However, the following three elements 
\begin{align*}
& (\rho_1, \dots, \rho_{k-1}, \alpha_k, \rho_{k+1}, \dots, \rho_m), \\ 
& (\alpha_1, \dots, \alpha_{k-1}, 1-\alpha_k, 1-\rho_{k+1}, \dots, 1-\rho_m)\in T_{(\alpha_1, \dots, \alpha_{k-1}, 1-\alpha_k)}, \\ 
& (1-\alpha_1, 1-\alpha_2, \dots, 1-\alpha_{k-1}, 1-\alpha_{k}, 1-\rho_{k+1}, \dots, 1-\rho_m) \in T_{(1-\alpha_1, 1-\alpha_2, \dots, 1-\alpha_k)}
\end{align*}
are in $T$, but there is no coordinate in which all three words agree, contradicting the hypothesis that $T$ is a $3$-wise intersecting family. This completes the proof that $T_{\boldsymbol{\delta}}\neq\emptyset$ for each $\boldsymbol{\delta}\in\{0,1\}^k$.

Second, we show that for each $\boldsymbol{\delta}=(\delta_1, \dots, \delta_k)\in\{0,1\}^k$, 
\[
\widetilde{T_{\boldsymbol{\delta}}}=\{\widetilde{\mathbf{x}}=(x_{k+1}, \ldots, x_m)\in \{0,1\}^{m-k} \ \mid \ (\delta_1, \delta_2, \ldots, \delta_k, x_{k+1}, \ldots, x_m)\in T_{\boldsymbol{\delta}}\}
\] 
is intersecting. Given $\widetilde{\mathbf{x}}=(x_{k+1}, \dots, x_{m}), \widetilde{\mathbf{y}} = (y_{k+1}, \dots, y_{m}) \in \widetilde{T_{\boldsymbol{\delta}}}$, pick an arbitrary element 
$$
\mathbf{z} = (1-\delta_1, 1-\delta_2, \dots, 1-\delta_k, z_{k+1},\dots, z_m ) \in T_{(1-\delta_1, 1-\delta_2, \dots, 1-\delta_{k})}.
$$
Since $T$ is $3$-wise intersecting, the three binary words $(\boldsymbol{\delta},\widetilde{\mathbf{x}}), (\boldsymbol{\delta}, \widetilde{\mathbf{y}})$, and $\mathbf{z}$ intersect, that is, they agree in some coordinate. This shared bit must be in a position indexed by some $j\in \{k+1, \dots, m\}$. In particular, $\widetilde{\mathbf{x}}$ and $\widetilde{\mathbf{y}}$ intersect; as a result, $\widetilde{T_{\boldsymbol{\delta}}}\subseteq \{0,1\}^{m-k}$ is intersecting. By Lemma~\ref{lem:weak}, 
\begin{equation}\label{eq:binary-vector-claim-upper-bound}
|T_{\boldsymbol{\delta}}|=|\widetilde{T_{\boldsymbol{\delta}}}|\leq 2^{m-k-1}.
\end{equation}
Partition
\begin{equation}\label{eq:binary-vector-claim-size}
T_{(\delta_1, \delta_2, \ldots, \delta_{k-1})}=T_{(\delta_1, \delta_2, \ldots, \delta_{k-1},\delta_k)} \cup T_{(\delta_1, \delta_2, \ldots, \delta_{k-1},1-\delta_k)}.
\end{equation}
By inductive hypothesis, $|T_{(\delta_1, \delta_2, \ldots, \delta_{k-1})}|= 2^{m-k}$. Combining \eqref{eq:binary-vector-claim-upper-bound} and \eqref{eq:binary-vector-claim-size} finishes the proof of the claim. \end{poc}

We now apply Claim~\ref{claim} (which is valid since $T$ is assumed not to be a star) to complete the proof by contradiction. Setting $k=m-1$ and $\boldsymbol{\delta}=(0,0,\dots,0)\in\{0,1\}^{m-1}$, we get $|T_{\boldsymbol{\delta}}|=1$, so $T$ contains a unique word of the form $\mathbf{x}=(0,0,\dots,0,u)$ for some $u\in\{0,1\}$. Setting $\boldsymbol{\delta}=(1,1,\dots,1)$, we similarly find a unique word $\mathbf{y}=(1,1,\dots,1,v)\in T$. Since $\mathbf{x}$ and $\mathbf{y}$ agree only at position $m$, and $T$ is intersecting, we must have $u=v$.

Suppose some $\mathbf{z}\in T$ has last coordinate $1-u$. Then $\mathbf{x}$, $\mathbf{y}$, and $\mathbf{z}$ share no common coordinate: any common coordinate would have to be position $m$ (the only position where $\mathbf{x}$ and $\mathbf{y}$ agree), but $\mathbf{z}$ takes the value $1-u\neq u$ there. This contradicts the $3$-wise intersecting hypothesis. As a result, every word in $T$ has last coordinate $u$. Since the family of all binary words with last coordinate $u$ is a star of size $2^{m-1}=|T|$, it follows that $T$ equals this star; this contradicts our assumption that $T$ is not a star.
\end{proof}

\section{A new proof of Theorem~\ref{thm:alphabet-intersecting}}\label{sec:q>2}

We outline the strategy before giving the details. The proof proceeds by induction on $m$. The key idea is to partition $T$ by the first letter to produce subfamilies $\widetilde{T}_1, \ldots, \widetilde{T}_q$ of $(m-1)$-letter words. A double-counting argument, together with Claim~\ref{claim:upper-bound-alphabet} below, forces one of two scenarios for each set $S(\boldsymbol{\delta})$: either it is contained entirely in a single $\widetilde{T}_i$, or it meets each $\widetilde{T}_i$ in the same singleton. In the first scenario, the hypothesis $q\geq 3$ allows us to propagate this containment to all of $\Z_q^{m-1}$, showing that $T$ is a star. In the second scenario, all subfamilies $\widetilde{T}_i$ coincide, and their common size and intersecting property let us apply the inductive hypothesis directly.

\begin{proof}[Proof of Theorem~\ref{thm:alphabet-intersecting}]
We proceed by induction on $m$. The base case $m=1$ is trivial. We may assume, without loss of generality, that the alphabet $X$ is given by $\Z_q\colonequals \mathbb{Z}/q\mathbb{Z}$. 

Since $T$ is a maximum intersecting family, $|T|=q^{m-1}$ by Lemma~\ref{lem:weak}. We partition the words in $T$ according to the first letter by defining 
\[
T_i=\{\mathbf{x}=(x_1, \dots, x_m)\in T \ \mid \ x_1=i\}, \qquad \text{and} \qquad \widetilde{T}_i=\{\widetilde{\mathbf{x}}=(x_2, \dots, x_m) \ \mid \  (i,\widetilde{\mathbf{x}})\in T_i\}
\]
for each $1\leq i \leq q$.

Given $\boldsymbol{\delta}\in \Z_q^{m-1}$, consider the set of $q$ words given by $S(\boldsymbol{\delta})=\{\boldsymbol{\delta}, \boldsymbol{\delta}+\mathbf{1}, \dots, \boldsymbol{\delta}+(\mathbf{q-1})\}$ where $\mathbf{c}=(c, \dots, c)$ represents a constant vector in $\Z_q^{m-1}$. There are $q^{m-1}$ choices of $\boldsymbol{\delta}\in \Z_q^{m-1}$. Since each set $S(\boldsymbol{\delta})$ has $q$ elements, this yields $q^{m-2}$ distinct sets of the form $S(\boldsymbol{\delta})$. Our key observation is the following claim. 
\begin{claim}\label{claim:upper-bound-alphabet}
For each such $S(\boldsymbol{\delta})$, we have:
\[
\sum_{i=1}^{q} |\widetilde{T}_i\cap S(\boldsymbol{\delta})| \leq q.
\]
Moreover, the equality holds if and only if $S(\boldsymbol{\delta}) \subseteq \widetilde{T}_i$ for a unique $1\leq i \leq q$, or there exists $\mathbf{k}=(k, \dots, k)$ such that $S(\boldsymbol{\delta}) \cap \widetilde{T}_i=\{\boldsymbol{\delta}+\mathbf{k}\}$ for all $1\leq i \leq q$.
\end{claim}
\begin{poc}
Fix $\boldsymbol{\delta}\in \Z_q^{m-1}$. We consider two cases.

\textbf{Case 1.} There exists some $1\leq i \leq q$ such that $|S(\boldsymbol{\delta})\cap \widetilde{T}_i|\geq 2$.

Take $\boldsymbol{\delta}+\mathbf{c_1}, \boldsymbol{\delta}+\mathbf{c_2}\in S(\boldsymbol{\delta})\cap \widetilde{T}_i$ with $c_1\neq c_2$. If $\boldsymbol{\delta}+\mathbf{c_3}\in S(\boldsymbol{\delta})\cap \widetilde{T}_j$ for $j\neq i$, then $\mathbf{c_3}\neq \mathbf{c_1}$ or $\mathbf{c_3}\neq \mathbf{c_2}$; however, $(i,\boldsymbol{\delta}+\mathbf{c_1})$ or $(i,\boldsymbol{\delta}+\mathbf{c_2})$ does not intersect $(j,\boldsymbol{\delta}+\mathbf{c_3})$. This shows $S(\boldsymbol{\delta})\cap \widetilde{T}_j=\emptyset$ for each $j\neq i$. Therefore,
\[
\sum_{j=1}^{q} |\widetilde{T}_j\cap S(\boldsymbol{\delta})| = |\widetilde{T}_i\cap S(\boldsymbol{\delta})| \leq |S(\boldsymbol{\delta})| = q,
\]
with the equality holding if and only if 
$S(\boldsymbol{\delta})\subseteq \widetilde{T}_i$.

\textbf{Case 2.} For all $1\leq i\leq q$, we have $|S(\boldsymbol{\delta})\cap \widetilde{T}_i|\leq 1$.

It follows that $\sum_{i=1}^{q}|\widetilde{T}_i\cap S(\boldsymbol{\delta})|\leq q$, with the equality holding if and only if
$|S(\boldsymbol{\delta})\cap \widetilde{T}_i|=1$ for $1\leq i \leq q$. Suppose the equality holds. If these singleton sets $S(\boldsymbol{\delta})\cap \widetilde{T}_i$ do not all coincide, then a similar argument as in the previous paragraph leads to a contradiction. Hence, we can find the same constant $\mathbf{k}$ for which $S(\boldsymbol{\delta}) \cap \widetilde{T}_i=\{\boldsymbol{\delta}+\mathbf{k}\}$ for all $1\leq i\leq q$.
\end{poc}

We count triples $(\boldsymbol{\delta},j,\mathbf{w})$ with $\mathbf{w}\in \widetilde{T}_j\cap S(\boldsymbol{\delta})$ in two different ways. First, summing $|\widetilde{T}_j\cap S(\boldsymbol{\delta})|$ over all $\boldsymbol{\delta}\in \Z_q^{m-1}$ and all $1\leq j\leq q$ counts these triples by fixing $\boldsymbol{\delta}$ and $j$, and then counting the admissible choices of $\mathbf{w}$. Second, for each fixed $\mathbf{w}\in \widetilde{T}_j$, there are exactly $q$ choices of $\boldsymbol{\delta}\in \Z_q^{m-1}$ for which $\mathbf{w}\in S(\boldsymbol{\delta})$, namely $\boldsymbol{\delta}=\mathbf{w}-\mathbf{c}$ with $c\in \Z_q$. Since $\sum_{j=1}^q |\widetilde{T}_j|=|T|$, it follows that
\[
\sum_{\boldsymbol{\delta}\in \Z_q^{m-1}} \sum_{j=1}^{q} |\widetilde{T}_{j}\cap S(\boldsymbol{\delta})| = q|T| = q^{m}.
\]
The outer sum has $q^{m-1}$ terms, and each inner summand is at most $q$ by Claim~\ref{claim:upper-bound-alphabet}. Thus, for equality to hold, it must be the case that
\begin{equation}\label{eq:equality-alphabet} \sum_{i=1}^{q} |\widetilde{T}_i\cap S(\boldsymbol{\delta})| = q
\end{equation}
for each $\boldsymbol{\delta}\in \Z_q^{m-1}$. Therefore, by the ``moreover" part of Claim~\ref{claim:upper-bound-alphabet}, for each $\boldsymbol{\delta}\in \Z_q^{m-1}$, either $S(\boldsymbol{\delta})\subseteq \widetilde{T}_{i}$ for a unique choice of $i\in \{1, 2, \dots, q\}$, or there exists $k$ (which might depend on $\boldsymbol{\delta}$), such that $S(\boldsymbol{\delta})\cap \widetilde{T}_{i}$ consists of the same single element $\boldsymbol{\delta}+\mathbf{k}$ for all $i$. 

Next, we consider these two cases separately.

\smallskip

\textbf{Case 1:} there exist $1\leq i \leq q$ and $\boldsymbol{\delta} \in \Z_q^{m-1}$ such that $S(\boldsymbol{\delta}) \subseteq \widetilde{T}_i$.

We claim that $\widetilde{T}_i = \Z_q^{m-1}$ so that $T$ is a star consisting of all words whose first coordinate is $i$.

Let $\mathbf{u}\in \Z_q^{m-1}$. Since $q\geq 3$, we can find $\mathbf{v}, \mathbf{v'}\in (\Z_q\setminus \{0\})^{m-1}$ such that $\mathbf{u}=\boldsymbol{\delta}+\mathbf{v}+\mathbf{v'}$. Indeed, any element in $\Z_q=\{0, 1, \dots, q-1\}$ can be expressed as a sum of two nonzero elements: $j=1+(j-1)$ for $j\neq 1$, and $1=2+(q-1)$.

Observe that for each $\mathbf{w}\in S(\boldsymbol{\delta})$ and each $1\leq j \leq q$ with $j\neq i$, we have $\mathbf{w}+\mathbf{v} \notin \widetilde{T}_j$; for otherwise $(i,\mathbf{w})$ and $(j, \mathbf{w}+\mathbf{v})$ are both in $T$ but do not intersect. This shows that $(S(\boldsymbol{\delta})+\mathbf{v})\cap \widetilde{T}_{j}=\emptyset$ for each $j\neq i$, and thus Claim~\ref{claim:upper-bound-alphabet} forces $S(\boldsymbol{\delta})+\mathbf{v} \subseteq \widetilde{T}_i$. Consequently, $S(\boldsymbol{\delta}+\mathbf{v})\subseteq \widetilde{T}_i$. Now we apply the same argument with $\boldsymbol{\delta}$ replaced with $\widetilde{\boldsymbol{\delta}}=\boldsymbol{\delta}+\mathbf{v}$ to deduce that $S(\widetilde{\boldsymbol{\delta}}+\mathbf{v'})\subseteq \widetilde{T}_i$. In particular, $\mathbf{u}=\boldsymbol{\delta}+\mathbf{v}+\mathbf{v'}=\widetilde{\boldsymbol{\delta}}+\mathbf{v'}\in \widetilde{T}_i$. Since $u\in \Z_q^{m-1}$ is arbitrary, we have shown that $\widetilde{T}_i = \Z_q^{m-1}$, as desired.

\textbf{Case 2:} for \emph{each} $\boldsymbol{\delta}\in \Z_q^{m-1}$, there is $k=k(\boldsymbol{\delta})$ such that 
\[
S(\boldsymbol{\delta}) \cap \widetilde{T}_1 = S(\boldsymbol{\delta}) \cap \widetilde{T}_2 =\dots = S(\boldsymbol{\delta}) \cap \widetilde{T}_q = \{\boldsymbol{\delta} + \mathbf{k}\}.
\]

It follows that $\widetilde{T}_1=\widetilde{T}_2=\dots=\widetilde{T}_q$; write this common family as $U$. We claim that $U$ is intersecting. Indeed, if $\mathbf{u},\mathbf{v}\in U$, then $(1,\mathbf{u})$ and $(2,\mathbf{v})$ are both in $T$, so since $T$ is intersecting and these two words differ in the first coordinate, they must agree in one of the last $m-1$ coordinates. Hence $\mathbf{u}$ and $\mathbf{v}$ intersect. Moreover, $|U|=q^{m-2}$. By the inductive hypothesis, $U$ is a star in $\Z_q^{m-1}$, and therefore $T$ is a star in $\Z_q^m$.

\smallskip

In both cases, $T$ is a star. This completes the proof.
\end{proof}

\section*{Acknowledgments}
We are grateful to Karen Meagher for helpful discussions. We are also grateful to the anonymous referee for their valuable comments. The second author also thanks Santa Clara University for its hospitality during his visit, during which this project was initiated. 

\bibliographystyle{abbrv}
\bibliography{main}

\end{document}